\documentclass[preprint,12pt]{amsart}
\usepackage{amssymb,amsthm,amsmath,graphicx}
\usepackage[utf8]{inputenc}

\newtheorem{theorem}{Theorem}
\newtheorem*{theorem*}{Theorem}
\newtheorem*{theoremA}{Theorem A}
\newtheorem*{theoremB}{Theorem B}
\newtheorem{definition}{Definition}
\newtheorem{lemma}{Lemma}
\newtheorem{remark}{Remark}
\newtheorem{corollary}{Corollary}
\newtheorem{example}{Example}

\begin{document}

\author{Douglas Finamore}
\address{Departament of Mathematics - ICMC, University of S\~ao Paulo}
\email{douglas.finamore@usp.br}
\thanks{
    This work was funded by Brazilian Coordena\c c\~ao de Aperfei\c coamento de Pessoal de N\'ivel Superior (CAPES), grant PROEX-11377206/D.
    I would like to express my deep appreciation to Eugenia Loiudice and Carlos Maquera for their contributions and insightful discussions throughout the course of this project. 
    Their input has been instrumental in its development.
}



\dedicatory{}

\begin{abstract}

We show that every quasiconformal contact foliation supports an invariant metric and characterise these foliations by the dynamical property of $C^1$-equicontinuity. We prove that a generalisation of the Weinstein conjecture holds for quasiconformal contact foliations, and provide a lower bound to the number of closed leaves. In particular, we show that the Weinstein conjecture holds for quasiconformal Reeb fields.

\end{abstract}

\title{Quasiconformal contact foliations}
\date{January 2023}

\maketitle

\section{Introduction}

\par A $q$-contact structure is a generalisation of a number of contact-like structures, including contact manifolds. 
On a $(2n+q)$-dimensional manifold $M$, such structure induces a splitting $TM = \xi \oplus \mathcal{R}$, where $\xi$ is a totally non-integrable bundle and $\mathcal{R}$ integrates to a \emph{contact foliation} $\mathcal{F}$, which is the orbit foliation of a Lie group action of $\mathbb{R}^q$ on $M$, called a \emph{contact action}. 
In this respect, $\mathcal{F}$ is a high dimensional analogue to the flow of the Reeb vector field of a coorientable contact structure. 
For instance, Weyl chamber actions are of contact nature \cite[Theorem 1.0.2]{almeida_contact_2018}.
\par In the author's earlier work \cite{finamore_contact_2022}, contact foliations were investigated in further detail, with particular interest in the existence of closed leaves when the ambient manifold is closed.
We called such problem the \emph{strong generalised Weinstein conjecture}. 
The conjecture holds in several cases, particularly when $\mathbb{R}^q$ acts on $M$ via isometries of a metric $g$. In this case, we say the contact foliation is \emph{isometric.}
\par It turns out that the existence of a bundle-like metric for $\mathcal{F}$ is sufficient to guarantee that $\mathcal{F}$ is isometric. 
In other worlds, every Riemannian contact foliation is isometric, so the two classes coincide \cite[Section 3.2]{finamore_contact_2022}. 
In this paper, we strengthen this equivalence by showing that, up to a choice of metric, a quasiconformal (in particular, conformal) contact foliation can be assumed isometric. 
More accurately, for contact foliations all these different notions can be characterised by the property of being equicontinuous with respect to the $C^1$-topology, which we call \textit{$C^1$-equicontinuity}.
\begin{theoremA}
    Let $\mathcal{F}$ be a contact foliation on a closed manifold $M$. The following are equivalent.
    \begin{itemize}
        \item[(i)] $\mathcal{F}$ is $C^1$-equicontinuous;
        \item[(ii)] the $C^1$-enveloping group $\mathrm{E}^1_{F}$ is a torus;
        \item[(iii)] $\mathcal{F}$ admits an invariant metric;
        \item [(iv)] $\mathcal{F}$ admits a bundle-like metric;
        \item[(v)] $\mathcal{F}$ is quasiconformal;
        \item[(vi)] $\mathcal{F}$ admits an invariant conformal structure.
    \end{itemize}
\end{theoremA}
\par In the second half of this paper, using the existence of invariant metrics, we construct a Morse-Bott function which can be used to count the closed leaves of a quasiconformal contact foliation.
This will allow us to conclude the validity of the strong Weinstein conjecture for quasiconformal contact foliations.
In particular, this shows that every quasiconformal Reeb field on a closed contact manifold satisfies the Weinstein conjecture, a result not previously known, to the best of the author's knowledge.
\begin{theoremB}
    A quasiconformal contact foliation of codimension $2n$ on a closed manifold $M$ has at least $n+1$ closed leaves.
\end{theoremB}
\par We can actually provide a better lower bound, depending on specific properties of the contact foliation and on the topology of the ambient manifold.
To this end, we take inspiration from previously known results in Contact Dynamics regarding \emph{K-contact manifolds} \cite{rukimbira_topology_1995} and \emph{metric $f$-K-contact manifolds} \cite{goertsches_topology_2020}.
These are particular types of $f$-structure (in the sense of Yano \cite{yano_structure_1982}) whose automorphisms preserve a specific metric tensor, called a contact metric. 
In order to qualify as a contact metric, a metric tensor must satisfy several compatibility properties concerning the tensor $f$. 
What we discovered is that, surprisingly, the properties that allowed the construction of the Morse-Bott function in
\cite{rukimbira_topology_1995, goertsches_topology_2020} do not really depend on the tensor $f$ or the contact metric, but only on the fact the Reeb fields are \emph{Killing} (for \emph{any metric}). 
We were thus capable of extending the construction of the Morse-Bott function to the much more general class of isometric $q$-contact structures.

\subsection{Structure of the paper}
Section \ref{prelim} discusses the basics of contact foliations and isometric contact actions. 
In Section \ref{equicon}, we define the notions of $C^1$-equicontinuity, conformality and quasiconformality and show how all such notions are equivalent. 
Finally, in Sections \ref{morseTheory} and \ref{cohom}, we relate the basic cohomology of a $C^1$-equicontinuous contact foliation to its closed orbits and the topology of the ambient manifold. 

\section{Preliminaries}\label{prelim}

\subsection{Contact foliations}
\par A contact foliation is a generalisation of the flow of the Reeb vector field determined by a co-orientable contact structure.
The analogue structure determining the contact foliation is called a $q$-contact structure, defined as follows.

\begin{definition}[\emph{$q$-contact manifolds} \cite{almeida_contact_2018}]\label{qcontactstructure}
    Let $n,q$ be positive integers and consider a $2n+q$ dimensional differential manifold $M$. 
    A \textbf{$q$-contact structure} on $M$ is a collection $\vec{\lambda} = (\lambda_1, \cdots, \lambda_q)$ of $q$ linearly independent non-vanishing $1$-forms $\lambda_i$, together with a splitting 
    \begin{equation*}
        TM = \mathcal{R} \oplus \xi
    \end{equation*}
    of the tangent bundle, satisfying the following conditions:
    \begin{itemize}
        \item[(i)] $\xi := \cap_i \ker\lambda_i$;
        \item[(ii)] for every $i$, the restriction $d \lambda_i\rvert_\xi$ is non-degenerate;
        \item[(iii)] for every $i$, one has $\ker d \lambda_i = \mathcal{R}.$
    \end{itemize}
A manifold endowed with such structure is called a \textbf{$q$-contact manifold} and denoted by $(M, \vec{\lambda}, \mathcal{R} \oplus \xi)$, or simply by $M$ when the context permits. 
We call the collection $\{\lambda_i\}$ an \textbf{adapted coframe} for the $q$-contact structure, and the $q$-form
\begin{equation*}
    \lambda := \lambda_1 \wedge \cdots \wedge \lambda_q \neq 0
\end{equation*}
is called the \textbf{characteristic form}. 
The bundles $\mathcal{R}$ and $\xi$ are called the \textbf{Reeb distribution} and \textbf{$q$-contact distribution,} respectively.
\end{definition}

\par Such a structure defines a unique set of $q$ vector fields $R_i$ on $M$, called the \textbf{Reeb vector fields}, satisfying the following properties (cf. \cite{almeida_contact_2018, finamore_contact_2022}).
\begin{lemma}\label{ReebFields}
    There is a unique collection of linearly independent vector fields $R_1, \cdots, R_q$ tangent to $\mathcal{R}$ satisfying the relations 
    \begin{itemize}
        \item[(i)] $\lambda_i(R_j) = \delta_{ij}$;
        \item[(ii)] $[R_i, R_j] = 0$, for all $i, j = 1, \cdots, q$;
        \item[(iii)] $\mathcal{R} = \mathrm{Span}\{R_1, \cdots, R_q\}.$
    \end{itemize}
\end{lemma}

\par Since the Reeb fields are pair-wise commutative, Frobenius Theorem implies that $\mathcal{R}$ integrates to a $q$-dimensional foliation $\mathcal{F}$ on $M$, called the \textbf{contact foliation}.
Note that when $q=1$, the definition becomes that of a contact manifold, and the contact foliation is simply the flow-lines of the Reeb field.
\par More generally, item (ii) in Lemma \ref{ReebFields} above implies that, for $s=1, \cdots, q$, each of the bundles
\[
    \mathcal{R}_s := \mathrm{Span}\{R_1, \cdots, R_s\},
\]
is integrable, with underlying foliation $\mathcal{F}_s$. 
Of course, $\mathcal{R}_q = \mathcal{R}$ and $\mathcal{F}_q = \mathcal{F}$. 
Moreover, the leaves of $\mathcal{F}_{s+1}$ are submanifolds of $M$ foliated by the leaves of $\mathcal{F}_{s}$. For every $s$, the leaves of $\mathcal{F}_s$ can be thought of as the orbits of a locally free action of $\mathbb{R}^s$ on $M$, and therefore their topological type is $\mathbb{R}^l \times \mathbb{T}^{s-l}$ for some $0 \leq l \leq s$, where $\mathbb{T}^{s-l}$ is the $(s-l)$-dimensional torus. 
\par We say that the contact foliation $\mathcal{F}$ satisfies the \emph{weak generalised Weinstein conjecture} if it has a leaf which is not homeomorphic to a plane and that it satisfies the \emph{strong generalised Weinstein conjecture} if it has a toric leaf (cf. \cite{finamore_contact_2022}).
\par Following \cite{asuke_transversely_1998}, we say that a holonomy invariant transverse measure for a foliation $\mathcal{F}$ is \emph{good} if it is non-atomic and the union of all the leaves in its support is the entire ambient manifold $M$.    
\begin{lemma}\label{goodMeasure}
    The measure $\mu_i = \|d\lambda_i^n\|$ is a good measure for $\mathcal{F}$, for $i=1, \cdots, q$.
\end{lemma}
\begin{proof}
   Condition (ii) in Definition \ref{qcontactstructure} is equivalent to $d\lambda_i^n \neq 0$, and therefore $\mu_i$ as defined is a $\mathcal{F}$-invariant volume form on $\xi$. 
    The measure $\mu$ also have full support, because given a transverse section $T$, $x \in M$ and $y \in \mathcal{F}(x) \cap T$, any open set $U \subset T$ containing $y$ is such that $\mu_i(U) > 0$. Thus $\mathcal{F}(x) \in \mathrm{supp}(\mu_i)$, and since $x$ is arbitrary it follows that
    \begin{equation*}
        \bigcup_{\mathcal{F}(x) \in \mathrm{supp}(\mu_i)}\mathcal{F}(x) = M,
    \end{equation*}
    and therefore, $\mu_i$ is a good measure.
\end{proof}
\subsection{Invariant metrics}
\begin{definition}
    We say a contact foliation is \textbf{isometric} if the ambient manifold $M$ supports a metric for which every Reeb field is Killing.
    In other words, a metric on $M$ is invariant under the contact action.
\end{definition}

Recall that a bundle-like metric on a foliated manifold $(M, \mathcal{F})$ is a metric $g$ on $M$ whose restriction to the normal bundle of $\mathcal{F}$ is holonomy invariant.
As shown in \cite[Section 3.2]{finamore_contact_2022}, every bundle-like metric for a $q$-contact foliation can be modified in the Reeb bundle $\mathcal{R}$ to obtain a metric for which every Reeb field is Killing.
Thus, when dealing with closed orbits of Riemannian contact foliations, there is no loss of generality in restricting attention to the class of isometric foliations.

Isometric $q$-contact structures directly generalise some previously known contact-like structures.

\begin{example}[K-contact, Sasakian and metric $f$-K-contact manifolds]
    In a contact manifold $(M, \lambda)$, the pair $(\ker\lambda, d\lambda)$ is a symplectic bundle over $M$.
    Therefore it supports a compatible, almost complex structure $J$. If $R$ is the Reeb field, we extend $J$ to the whole $TM$ by setting $J(R) = 0$.
    The equations
    \begin{align*}
        g(JX, Y) &= d\lambda(X, Y), \\
        g(R, X) &= \lambda(X),
    \end{align*}
    define a Riemannian metric on $M$, called the \textbf{contact metric}.
    When $R$ is Killing with respect to the contact metric $g$, $M$ is said to be a \textbf{K-contact manifold}.
    If this structure satisfies the additional condition 
    \[
        (\nabla_XJ)Y = g(X, Y)R - \lambda(Y)X,
    \]
    the manifold $M$ is said to be \textbf{Sasakian}.
    If the Reeb field is Killing with respect to a metric $g$ (not necessarily a contact metric), then $M$ is said to be an \textbf{R-contact manifold}. (cf. \cite{BlairContactmanifoldsRiemannian1976, blair_riemannian_2010, rukimbira_remarks_1993, rukimbira_topology_1995}).
    For instance, every Brieskorn manifold is Sasakian \cite[Chapter 5, Section 7]{kon_structures_1985}.

    A metric $f$-K-contact manifold generalises the concept of a K-contact manifold in that it allows for more than one Reeb field.
    Precisely, a metric $f$-K-contact manifold $(M, f, g, R_i, \lambda_i)$ consists of an $f$-manifold (cf. \cite{yano_structure_1982}) together with a metric $g$, Killing fields $R_1, \cdots, R_q$ and $1$-forms $\lambda_1, \cdots, \lambda_q$ satisfying the following compatibility conditions:
    \begin{itemize}
        \item[-] $\lambda_i(R_j) = \delta_{ij}$;
        \item[-]  $d\lambda_i(X, Y) = g(fX, Y) \text{ for all } i$; 
        \item[-] $f(R_i) = 0 \text{ for all } i$;       
        \item[-] $\mathrm{Im}f = \cap_i\ker\lambda_i$; 
        \item[-] $(f\rvert_{\mathrm{Im}f})^2 = -\mathrm{id}\rvert_{\mathrm{Im}f}$; 
        \item[-] $g(fX, fY) = g(X,Y) - \sum_i\lambda_i(X)\lambda_i(Y)$.   
    \end{itemize}
    In particular, there is a splitting $TM = \ker f \oplus \mathrm{Im}f$, and the fields $R_i$ span $\ker f$ and are commutative.
    One way to obtain such structures is to consider the mapping tori of automorphisms of K-contact manifolds preserving the K-contact structure (cf. \cite{goertsches_topology_2020}).   
\end{example}

We remark that, in an isometric $q$-contact manifold, there are no compatibility assumptions on the metric $g$; we ask simply that the Reeb fields are all Killing.
Furthermore, note that, in a metric $f$-K-contact structure, the exterior derivatives $d\lambda_i$ are all equal, while in an $q$-contact structure they need not be the same but only have the same kernels. 

\begin{definition}
    A $q$-contact structure defined by an adapted coframe $\{\lambda_1, \cdots, \lambda_q\}$ is said to be \textbf{uniform} if
    \[
        d\lambda_i = d\lambda_j,
    \]
    for every $i, j$.
\end{definition}

The isometric $q$-contact structures defined by metric $f$-K-contact manifolds are uniform.
In general, a $q$-contact structure need not be uniform, as is the case for the $q$-contact structure on $\mathrm{SO}(q, q+n)/\mathrm{SO}(n)$ constructed by Almeida in \cite[Section 3.3.1]{almeida_contact_2018} (for remarks on the different nomenclature used by Almeida in his work, see \cite{finamore_contact_2022}).
Another example of a non-uniform $q$-contact structure is the following.

\begin{example}[Product manifolds]\label{product}
    Suppose $(M_1, \alpha_1)$ and $(M_2, \alpha_2)$ are contact manifolds, with Reeb fields $S_1$ and $S_2$, respectively, and let $M$ be the product manifold $M_1 \times M_2$.
    Then $M$ supports a non-uniform $2$-contact structure.
    Indeed, using the canonical projections $\pi_i: M \to M_i$, for $i=1,2$, let
    \begin{align*}
        X_i &:= \pi_i^\ast S_i, \\
        \sigma_i &:= \pi_i^\ast\alpha_i, \\
        \xi_i &:= \pi_i^\ast\ker\alpha_i.
    \end{align*}
    We define on $M$
    \begin{align*}
        \lambda_+ &:= \sigma_1 + \sigma_2 \\
        \lambda_- &:= \sigma_1 - \sigma_2 
    \end{align*}
    It is clear that $\lambda_+$ and $\lambda_-$ are linearly independent.
    As 
    \[ 
        \ker\lambda_\pm = \xi_1\cap\xi_2 \oplus \mathrm{Span}\{X_1 \mp X_2\}, 
    \] it is immediate that 
    \[ 
        \xi := \ker\lambda_1 \cap \ker\lambda_2 = \xi_1\cap\xi_2.
    \]
    Moreover, if we let 
    \begin{align*}
        R_+ &:= \frac{1}{2}( X_1 + X_2), \\
        R_- &:= \frac{1}{2}( X_1 - X_2)
    \end{align*}
    then $\lambda_i(R_j) = \delta_{ij}$ and it is easy to check that $\mathcal{R} := \mathrm{Span}\{R_+, R_-\}$ satisfies 
    \[
        TM = TM_1 \oplus TM_2 = \mathcal{R} \oplus \xi. 
    \]
    It remains to show that the derivatives $d\lambda_i$ are non-degenerate on $\xi$. 
    Given $Y$ tangent to $\xi$, let $p = (p_1, p_2)$ be a point in $M$, and write $Y = Y_1 \oplus Y_2,$ with $Y_i \in \xi_i \cap TM_i$.
    The equality $d\lambda_i\rvert_p(Y_p, \cdot) \equiv 0$ implies 
    \[
        d\alpha_1\rvert_{p_1}(d\pi_1Y_1\rvert_{p_1}, \cdot) = \pm d\alpha_2\rvert_{p_2}(d\pi_2Y_2\rvert_{p_2}, \cdot),
    \]
    or, in more explicit terms, that for any choice of $Z_i \in T_{p_i} M_i$, we have 
    \[
        d\alpha_1\rvert_{p_1}(d\pi_1Y_1\rvert_{p_1}, Z_1) = \pm d\alpha_2\rvert_{p_2}(d\pi_2Y_2\rvert_{p_2}, Z_2).
    \]
    This can only happen when $d\pi_1Y_1\rvert_{p_1} = d\pi_2Y_2\rvert_{p_2} = 0$, for both $\alpha_1$ and $\alpha_2$ are non-degenerate. Hence 
    \[
        \iota_{Y}d\lambda_i \equiv 0 \iff Y = 0,
    \]
    and the $d\lambda_i$ are non-degenerate on $\xi,$ as we wished. 
    Note that, because the fields $X_1$ and $X_2$ are $\pi_1$-related to $S_1$ and $0$, respectively, it follows that $[X_1,X_2] = [S_1, 0] = 0$, hence $[R_+, R_-] = 0$, and by uniqueness we conclude that ${R_+, R_-}$ are indeed the Reeb fields of the contact action. 
    The Reeb distribution $\mathcal{R}$ can be seen as the span of both $\{R_+, R_-\}$ and $\{X_1, X_2\}$, hence its integral submanifolds are exactly the products of flow-lines of $M_1$ and $M_2$.

    If, moreover, there are metrics $g_i$ on $M_i$ with respect to which the fields $S_i$ are Killing, then 
    \[
        g := \pi_1^\ast g_1 + \pi_2^\ast g_2
    \]
    is a Riemannian metric on $M$ for which the fields $R_+$ and $R_-$ are Killing, so the product of R-contact manifolds is an isometric $2$-contact manifold. 
\end{example}

This construction can be generalised to the product of two $q$-contact structures whose Reeb distributions have the same rank. 
 If $M_1$ has as adapted coframe the collection $\{\lambda_1, \cdots, \lambda_q\}$ and $M_2$ the collection $\{\eta_1, \cdots, \eta_q\}$, then $M_1\times M_2$ admits a $2q$-contact structure with adapted coframe given by the forms
    \begin{align*}
        \lambda_i &:= \pi_1^\ast\alpha_i + \pi_2^\ast\beta_i, \\
        \eta_i &:= \pi_1^\ast\alpha_i - \pi_2^\ast\beta_i, 
    \end{align*}
    for $i = 1, \cdots, q$, and splitting 
    \[
        T(M_1\times M_2) = (\pi_1^\ast\mathcal{R}_1 \oplus \pi_2^\ast\mathcal{R}_2) \oplus (\pi_1^\ast\xi_1 \oplus \pi_2^\ast\xi_2).
    \]

\section{Quasiconformality and equicontinuity}\label{equicon}
The existence of an invariant metric is a geometric property of the manifold. 
We want to describe such property by dynamical means. 
In order to do this, we consider a compactification of the group action. 
Note that the image $F(\mathbb{R}^q)$ of the contact action $F: \mathbb{R}^q \to \mathrm{Diff}(M)$ is the subgroup spanned by all the flows of the Reeb fields $R_i$. 
In other words,
\[
    F(\mathbb{R}^q) = \mathrm{Span}\left\{\exp{(tR_i)}; t \in \mathbb{R}, i = 1, \cdots, q\right\}.
\]
Now, $\mathrm{Diff}(M)$ is a Lie Group when equipped with the $C^1$ compact-open open topology $\tau_1$, its Lie Algebra being that of vector fields on $M$.
    
\begin{definition}[$C^1$-enveloping group]\label{envGrp} Let $F: \mathbb{R}^q \to \mathrm{Diff}(M)$ be an action.
The \textbf{$C^1$-enveloping group of $F$} is the closure 
    \[
        \mathrm{E}^1_{F} := \overline{F(\mathbb{R}^q)}
    \] 
in the Lie group $(\mathrm{Diff}(M), \tau_1)$.
    \end{definition}
\begin{definition}
    The action $F$ is said to be \textbf{$C^1$-equicontinuous} if its $C^1$-enveloping group $\mathrm{E}^1_{F}$ is compact.
 \end{definition} 
 
The $C^1$-enveloping group acts on the manifold $M$ in a natural way, and its orbits are exactly the closures of the leaves of $\mathcal{F}$.

\begin{lemma}\label{orbitsClosuresAreEnvelopingAction}  
    Given a $C^1$-equicontinuous action $F: \mathbb{R}^q \to \mathrm{Diff}(M)$ and $x \in M$, the orbit of $x$ under the action of $\mathrm{E}^1_{F}$ is exactly $\overline{\mathcal{F}(x)}.$
\end{lemma}

\begin{proof}
    Let $y \in \overline{\mathcal{F}(x)}$. Then there is a sequence $a_n \in \mathbb{R}^q$ such that $F^{a_n}(x) \to y$ and since $\mathrm{E}^1_{F}$ is compact, $F^{a_n} \to T \in \mathrm{E}^1_{F}$ (up to a sub-sequence). Hence $y = T(x)$ belongs to the orbit of $x$ under $\mathrm{E}^1_{F}$. Conversely, if $y$ is in the orbit of $x$ under $\mathrm{E}^1_{F}$, then there is $T = \lim F^{a_n} \in \mathrm{E}^1_{F}$ such that $y = T(x) = \lim F^{a_n}(x)$, and therefore $y \in \overline{\mathcal{F}(x)}$. 
\end{proof}

\par We remark that $C^1$-equicontinuity is a relatively strong dynamic condition.
It means the family $F(\mathbb{R}^q)$ is equicontinuous in the classic sense, and each of the families of derivatives $\{d F^a_x; a \in \mathbb{R}^q\}$ is equicontinuous, for every $x \in M$. 
As it turns out, being $C^1$-equicontinuous and being an isometric action are equivalent conditions.

 \begin{theorem}\label{C1EquicontinuityIsEquivalentToInvariantMetric}
     Let $F: \mathbb{R}^q \to \mathrm{Diff}(M)$ be a $C^1$-equicontinuous action. Then $M$ supports an $F$-invariant metric.
 \end{theorem}

 \begin{proof}
     Let $g_0$ be any Riemannian metric on $M$. By hypothesis, $\mathrm{E}^1_{F}$ is a compact Lie group.
     Let $\mu$ be a Haar measure defined on the Borel $\sigma$-algebra of $\mathrm{E}^1_{F}$. 
     For each $p \in M$ and $X, Y \in T_pM$ we define a function
     \[
     \begin{aligned}\label{integralFunction1}
         \mathrm{E}^1_{F} &\to \mathbb{R} \\
         e &\mapsto (e^\ast g_0)_p(X,Y) = g_0\rvert_{e(p)}(d e_pX, d e_pY).
     \end{aligned}
     \]
     Note that two elements of $\mathrm{E}^1_{F}$ are close in the $C^1$ topology if they are close in the compact-open topology and their derivatives are close as transformations.
     This, together with the fact that $g_0$ is smooth, implies that the function defined by the mapping \ref{integralFunction1} is continuous.
     We defined a metric tensor $g$ on $M$ by averaging the metric $g_0$:
     \[
        g_p(X,Y) = \int_{\mathrm{E}^1_{F}}(e^\ast g_0)_p(X,Y)d\mu(e).
     \]
 
 This is a Riemannian metric since it is smooth, bi-linear and positive-definite.
 Moreover, it is invariant under the action of $\mathrm{E}^1_{F}$, as for any $f \in \mathrm{E}^1_{F}$ one has
 \begin{align*}
     (f^\ast g)_p(X,Y) &= g_0\rvert{f(p)}(d f_pX, d f_pY) \\
     &= \int_{e\in\mathrm{E}^1_{F}}e^\ast g_0\rvert_{f(p)}(d f_pX, d f_pY) \\
     &= \int_{e\in\mathrm{E}^1_{F}}(ef)^\ast g_0\rvert_p(X,Y) \\
     &= g_p(X,Y),
 \end{align*}
 thus $f^\ast g = g$.
 In particular, $(F^a)^\ast g = g$ for every $a \in \mathbb{R}^q,$ that is, $F$ is an isometric contact action.
\end{proof}

Thus, $C^1$-equicontinuity is simply a dynamical expression of the geometric property of preserving a metric. 

\par Every $C^1$-equicontinuous action is, in particular, equicontinuous and, therefore, strongly recurrent. 
To be more precise, the action $F$ is uniformly almost periodic \cite[Theorem 2.2]{auslander1988minimal} so that every leaf $\mathcal{F}(x)$ returns arbitrarily close to $x$ infinitely often. 
For $C^1$-equicontinuous contact actions, it was shown in \cite{finamore_contact_2022} that there are at least two actual closed orbits.
In Section \ref{morseTheory}, we will use Morse Theory to improve this lower bound.

\subsection{Quasiconformal and conformal contact foliations.}

Suppose $M$ is a $q$-contact manifold with a Riemannian metric $g$. At a point $x \in M$ and an element $a \in \mathbb{R}^q$, we consider the mappings 
\begin{align*}
    L_F(x, a) &:= \max\{ \lvert d F^a_x v \rvert ~;~ v\in \xi_ x \text{ and } \lvert v \vert = 1 \}, \\
    l_F(x, a) &:= \min\{ \lvert d F^a_x v \rvert ~;~  v\in \xi_ x \text{ and } \lvert v \vert = 1 \}.
\end{align*}
Together, these quantities define the $\xi$-\emph{eccentricity} of the action $F$ as the ratio
\[
E_F(x,a) := \frac{L_F(x, a)}{l_F(x,a)}.
\]
This map measures how much the unity ball in $\xi_x$ is deformed into an ellipsoid in $\xi_{F^a(x)}$ by the transformation $d F^a_x$.

\begin{definition}[Quasiconformal contact action]

We say the contact action $F$ is \textbf{quasiconformal} with respect to the metric $g$ if its $\xi$-eccentricity mapping
\[
E_F: M \times \mathbb{R}^q \to \mathbb{R}
\]
is globally bounded. Moreover, if $E_F \equiv 1$, we say the contact action is \textbf{conformal}.
\end{definition}

In the case of a contact manifold $(M, \lambda)$,  the action $F$ is simply the flow of the Reeb field $R_\lambda$, and we say that the field $R_\lambda$ is quasiconformal.
\par We remark that for compact $M$, quasiconformality does not depend on the metric since any two Riemannian metrics $g, g'$ are quasi-isometric. This is not true for conformality, as a change of metrics will generally change the function $E_F$. This is because quasiconformality is a \emph{dynamical} property, whereas conformality is mostly a \emph{geometric} one. Indeed, it is not hard to see that $E_F = 1$ for a metric $g$ if and only if there is a positive function $f: M\to\mathbb{R}$ such that $(F^a)^\ast g = fg$. In other words, the contact action $F$ supports an invariant conformal structure, i.e, a class $[g]$ of Riemannian metrics all of which are multiples of each other by positive functions.

\begin{lemma}\label{contactHolonomiesAreSymplectomorphisms}
    Let $T$ be a complete transversal for the contact foliation $\mathcal{F}$. If the action is conformal, then the holonomy pseudogroup of $\mathcal{F}$ with respect to $T$ consists of conformal transformations. 
\end{lemma} 

\begin{proof}
    Let $h: D(h) \to R(h)$ be a holonomy transformation of $\mathcal{F}$. We are going to show that every point of $D(h)$ has a neighbourhood restricted to which $h$ preserves the conformal structure $[g]$, and therefore $h: D(h) \to R(h)$ is a conformal transformation. Now, since $\mathcal{F}$ is the orbit foliation of an action $F: \mathbb{R}^q \times M \to M$, given $x \in D(h)$, there is an open set $U \subset D(h)$ containing $x$, and a function $\tau: U \to \mathbb{R}^q$, such that $h(u) = F(\tau(u), u)$ for every $u \in U$. However, $F$ is a composition of the flows of the Reeb field, each of which preserves the conformal structure $[g]$. Hence $F$, and therefore $h$, preserve the conformal structure $[g]$ on the open set $U$.
\end{proof}

\begin{corollary}\label{conformallyFlat}
    The underlying foliation of a conformal contact action is a \emph{transversely flat conformal foliation}, in the sense of \cite{asuke_transversely_1996}.
\end{corollary}

In general, in this order, Riemannian, conformal and quasiconformal foliations form a hierarchy of weaker properties. In the case of contact foliations, however, all these classes are the same, as the next results show.

\begin{lemma}\label{qscfImpliesCf}
     If a contact action $F$ on a compact manifold $M$ is quasiconformal, then there is a conformal structure $\gamma$ preserved by $F$.
\end{lemma}
\begin{proof}
    This is a theorem of Tukia \cite{tukia_quasiconformal_1986} for quasiconformal actions of general groups. 
    In the case of contact actions, the proof goes roughly as follows. 
    The space $C^\xi(x)$ of all conformal structures in $\xi_x$ can be endowed with a canonical metric making it into a non-positively curved space. Given a conformal structure $\tau_0 \in C^\xi(F^a(x))$, the subset
    \begin{equation*}
        C(x) := \{(dF_x^a)^{-1}(\tau_0(F^a(x))); a \in \mathbb{R}^q\}
    \end{equation*}
    is a bounded subset of $C^\xi(x)$, due to the quasiconformality of $F$. Since $C^\xi(x)$ is of non-positive curvature, this implies the existence of a unique ball of minimal radius containing $C(x)$. It can be shown that the centre of such a ball is an $F$-invariant conformal structure.
    
    In a different context, the same arguments were used by Sadovskaya to prove a similar result for quasiconformal Anosov actions (cf. \cite[Proposition 3.1]{sadovskaya_uniformly_2005})
\end{proof}

\begin{theoremA}\label{C1EquicontinuityCharacterisation}
    Let $\mathcal{F}$ be a contact foliation on a closed manifold $M$. The following are equivalent.
    \begin{itemize}
        \item[(i)] $\mathcal{F}$ is $C^1$-equicontinuous;
        \item[(ii)] the $C^1$-enveloping group $\mathrm{E}^1_{F}$ is a torus;
        \item[(iii)] $\mathcal{F}$ admits an invariant metric;
        \item [(iv)] $\mathcal{F}$ admits a bundle-like metric;
        \item[(v)] $\mathcal{F}$ is quasiconformal;
        \item[(vi)] $\mathcal{F}$ admits an invariant conformal structure.
    \end{itemize}
\end{theoremA}

\begin{proof}
First, let us show that (i), (ii), (iii) and (iv) are equivalent. 

\vspace{\baselineskip} \noindent \emph{(i) $\implies$ (ii)}: using Theorem \ref{C1EquicontinuityIsEquivalentToInvariantMetric}, we choose an invariant metric $g$ on $M$. 
By Myers-Steenrood theorem, \cite{myers_group_1939}, the space $\mathrm{Iso}(M,g)$ is a compact Lie group, hence $\mathrm{E}^1_{F}$ is a compact Abelian Lie subgroup, and therefore a torus.

\vspace{\baselineskip} \noindent\emph{(ii) $\implies$ (iii)}: this is a particular case of the averaging procedure of Theorem \ref{C1EquicontinuityIsEquivalentToInvariantMetric}. 

\vspace{\baselineskip} \noindent \emph{(iii) $\implies$ (iv)}: if (iii) holds, then the restriction of an invariant metric $g$ to the normal bundle $\xi$ is bundle-like; thus, (iv) also holds.

\vspace{\baselineskip} \noindent \emph{(iv) $\implies$ (i)}: if $g^\tau$ is a bundle-like metric on $\xi$, then 
\[
    g = g^\tau \oplus \sum_{i=1}^q(\lambda_i\otimes\lambda_i)
\]
is an invariant metric on $M$ (cf. \cite{finamore_contact_2022}). 
Thus $\mathrm{E}^1_{F}$ is a closed subspace of the Hausdorff compact space $\mathrm{Iso}(M,g)$. 
Therefore (i) is also true.
This shows the equivalence of the first four assertions. 
\par Now, we show that the last two items are equivalent to (iii). That (iii) implies (v) and (vi) is straightforward because the invariant metric $g$ defines an $\mathcal{F}$-invariant conformal class $[g]$. 
As for the remaining implications, we have

\vspace{\baselineskip} \noindent \emph{(v) $\implies$ (vi)}: this is Lemma \ref{qscfImpliesCf}.

\vspace{\baselineskip} \noindent \emph{(vi) $\implies$ (iii)}: if (vi) holds, then $\mathcal{F}$ is a transversely flat conformal foliation, as noted in Corollary \ref{conformallyFlat}. Now, the non-degenerate form $d\lambda_1$ induces a \emph{good measure} for the foliation $\mathcal{F}$, according to Lemma \ref{goodMeasure}. Every transversely conformal foliation supporting a good measure is, up to a  choice of metrics, a Riemannian foliation (cf. \cite[Theorem A]{asuke_transversely_1998}), and therefore $\mathcal{F}$ supports a bundle-like metric.
\end{proof}
\begin{example}\label{regularIsQsc}
    Suppose $(M, \lambda)$ is a regular contact manifold. This means each point of $M$ has a neighbourhood $U$ such that each flow-line of the Reeb field $R$ intersecting $U$ passes through $U$ a single time. If $M$ is closed, this means all the Reeb orbits are closed \cite{boothby_contact_1958}. It follows from Lemma \ref{orbitsClosuresAreEnvelopingAction} that the enveloping group of $\mathbb{R}$-action associated to the flow of $R$ is $S^1$, hence
    $R$ is a quasiconformal field.
\end{example}

In \cite[Theorem 1.1]{casals_geometric_2019}, the authors show that 
\[
    (\mathbb{R}^3 \times D^{2n-4}, \lambda_{\mathrm{ot}} + \omega_{\mathrm{std}})
\]
is a local model for any overtwisted contact manifold $(M, \alpha)$\footnote{
    Here $\lambda_{\mathrm{ot}} = \cos(\rho)d z + \rho\sin(\rho)d\theta$ is the standard overtwisted contact form on $\mathbb{R}^3$ with cylindrical coordinates $(\rho, \theta, z)$, and $\omega_{\mathrm{std}} = \frac{1}{2}\sum_i(x_i d  y_i - y_i d  x_i)$ is the standard symplectic form on $\mathbb{R}^{2n-4}$ with coordinates $(x_1, y_1, \cdots, x_{n-2}, y_{n-2})$.
}.
More precisely, they show that any overtwisted contact $(M, \alpha)$ manifold admits an contact embedding of $(\mathbb{R}^3 \times D^{2n-4}(r), \lambda)$, where $\lambda = \lambda_{\mathrm{ot}} + \omega_{\mathrm{std}}$, for some $r>0$, i.e. a map
\[
    f: (\mathbb{R}^3 \times D^{2n-4}(r), \lambda) \to (M, \alpha)
\]
such that $f^\ast\ker\alpha = \ker\lambda$.
Such contact embedding is not necessarily a strict one, that is, it does not need to satisfy $f^\ast\alpha = \lambda$.
As an application of Theorem A, we can show that for quasiconformal Reeb flows that is never the case.
\begin{theorem}\label{thm: otd&qsc}
    A contact form whose Reeb flow is quasiconformal does not admit a strict contact embedding of $(\mathbb{R}^3 \times D^{2n-4}, \lambda_{\mathrm{ot}} + \omega_{\mathrm{std}})$.
    In particular, any overtwisted contact structure admits a contact form whose associated Reeb field is not quasiconformal.
\end{theorem}
\begin{proof}
    We begin by looking at the particular case of the $3$-dimensional Euclidean space $(\mathbb{R}^3, \lambda_{\mathrm{ot}})$.
    We want to show its Reeb flow is not equicontinuous, and consequently not quasiconformal.
    In order to calculate its Reeb field $R_{\mathrm{ot}} = f\partial_\rho + g\partial_\theta + h\partial_z$, we write
    \begin{align*}
     d\lambda_{\mathrm{ot}} &= -\sin(\rho)d\rho\wedge d z + (\sin(\rho) +\rho\cos(\rho))d\rho\wedge d\theta \\
               &= d\rho\wedge\alpha,
    \end{align*}
    where $\alpha = (\sin(\rho) +\rho\cos(\rho))d\theta - \sin(\rho)d z$. 
    Then
    \begin{align*}
        0   &= d\rho\wedge\alpha(R_{\mathrm{ot}}, \cdot)\\
            &= d\rho(R_{\mathrm{ot}})\alpha - \alpha(R_{\mathrm{ot}})d\rho \\
            &= f[(\sin\rho +\rho\cos\rho)d\theta - \sin\rho d z] - [(\sin\rho +\rho\cos\rho)g - \sin\rho h]d\rho \\
            &= [\sin\rho h-(\sin\rho +\rho\cos\rho)g]d\rho + f(\sin\rho +\rho\cos\rho)d\theta - f\sin\rho d z.
    \end{align*}
    Together with $\lambda_{\mathrm{ot}}(R_{\mathrm{ot}}) = 1$, the above equations give rise to the the system
    \begin{equation*}
        \begin{cases}
            \cos\rho h + \rho\sin\rho g = 1 \\
            \sin\rho h = (\sin\rho +\rho\cos\rho)g \\
            f(\sin\rho +\rho\cos\rho) = 0 \\
            f\sin\rho = 0
        \end{cases}
    \end{equation*}
    which, for $\rho > 0$, has as solutions
    \begin{equation*}
        \begin{cases}
            f = 0 \\
            g = \frac{\sin\rho}{\rho + \sin(\rho)\cos(\rho)} = \frac{2\sin\rho}{2\rho + \sin(2\rho)} \\
            h = \frac{\sin\rho + \rho\cos\rho}{\rho + \sin(\rho)\cos(\rho)} = \frac{2\sin\rho + 2\rho\cos\rho}{2\rho + \sin(2\rho)},
        \end{cases}
    \end{equation*}
    where we made use of the relation $\sin(\rho)\cos(\rho) = \frac{\sin(2\rho)}{2}$. 
    Thus 
    \begin{equation*}
        R_{\mathrm{ot}} = \frac{2}{2\rho + \sin(2\rho))}\left(\sin(\rho)\partial_\theta + (\sin\rho + \rho\cos\rho)\partial_z\right).
    \end{equation*}
    Note that the field does not depend on $\theta$ or $z$, nor does it have any radial components. 
    This means each cylinder $\{\rho = \rho_0\}$ is invariant under the flow of $R_{\mathrm{ot}}$ and foliated by its flowlines, which are helices in $\mathbb{R}^3$.
    In particular, consider a value $\rho_0 > 0$ for which 
    \begin{equation*}
        \rho_0 + \tan(\rho_0) = 0.
    \end{equation*}
    Such value is an isolated zero of the function 
    \begin{equation*}
        \frac{2(\sin\rho + \rho\cos\rho)}{2\rho + \sin(2\rho))},
    \end{equation*}
    therefore the $\partial_z$ component of $R_{\mathrm{ot}}$ vanishes at the cylinder $C$ defined by $\rho = \rho_0$. 
    The restriction of the Reeb flow $\exp(tR_{\mathrm{ot}})$ to $C$ is a closed foliation by circles parallel to the $\rho\theta$-plane. 
    In particular, the $z$-component is constant along each flowline.
    \par Fix a point $x_0 = (\rho_0, \theta_0, z_0) \in C$. 
    Given $\epsilon > 0$, in any $\delta$-ball centred in $x_0$ we can find a point $y = (\rho_y, \theta_0, z_0)$ lying in a cylinder where the $\partial_z$ component of $R_{\mathrm{ot}}$ \emph{does not} vanish.
    This means the orbit of $y$ under the flow of $R_{\mathrm{ot}}$ is a helix, and therefore the $z$ component of $\exp(tR_{\mathrm{ot}})y$ grows in module as $t$ goes to infinity. 
    The $z$-component of $\exp(tR_{\mathrm{ot}})x_0$, on the other hand, is constant.
    Hence the distance between $\exp(tR_{\mathrm{ot}})y$ and $\exp(tR_{\mathrm{ot}})x_0$ will be greater than $\epsilon$ for large enough $t$, and the Reeb flow is not equicontinuous.     
    \par Next, in the general case, we consider the model
    \begin{equation*}
        (\mathbb{R}^3 \times D^{2n-4}, \lambda), 
    \end{equation*}
    noting that its Reeb field $R$ is simply $R_{\mathrm{ot}}+\vec{0}$, which is not equicontinuous by the same arguments as above.
   Now, given an arbitrary overtwisted contact manifold $(M, \alpha)$ whose Reeb field $R_\alpha$ is quasiconformal, if the embedding $f: (\mathbb{R}^3 \times D^{2n-4}(r), \lambda) \to (M, \alpha)$ were strict, the relation $f^\ast\alpha = \lambda$ would imply that $R_\alpha$ is not equicontinuous, a contradiction.
\end{proof}
Recalling the dicotomy overtwisted/tight in contact geometry, Theorem \ref{thm: otd&qsc} raises the question: is it true that every quasiconformal Reeb field is associated to a tight contact structure?
The converse affirmation is not true.
Indeed, the standard contact form on the $3$-torus is tight (even more, it is \emph{strongly fillable} \cite{eliashberg_unique_1996}) but it is not equicontinuous, and therefore not quasiconformal. 
Similar constructions hold in higher dimensions \cite{massot_weak_2013}.

\section{Counting closed orbits}\label{morseTheory}
In this section, we use Morse Theory to relate the number of closed leaves in an isometric contact foliation to its basic cohomology.
To this end, we begin by constructing a suitable Morse-Bott function on $M$, whose critical set is the collection of closed leaves.
Our arguments follow the same ideas from \cite{banyaga_note_1990, goertsches_topology_2020}, but without any condition on the metric $g$ other than its invariance under the contact action $F$.
\par Suppose $\mathcal{F}$ is a $C^1$-equicontinuous foliation on $M$, and let $g$ be an invariant metric.
Since each Reeb field is Killing, we can define the following:
\begin{definition}[Associate tensor field]\label{associatedTensorField} 
    For each Reeb field $R_i$ of an isometric contact foliation $\mathcal{F}$, its \emph{associate tensor field} is the $(1,1)$-tensor
    \begin{equation}
        f_i: X \mapsto \nabla_XR_i.
    \end{equation} 

\end{definition}
Using Koszul's formula, we see that these tensors satisfy
\begin{equation}\label{KosRev}
    g(f_i(X), Y) = \frac{1}{2}d\lambda_i(X,Y).
\end{equation}

\begin{lemma} For every $i$, the associated tensor field $f_i$ is such that
    \begin{itemize}
        \item[(i)] $f_i$ is skew-symmetric with respect to $g$;
        \item[(ii)] $f_i(R_j) = 0$ for every $j$;
        \item[(iii)] $\mathrm{Im}f_i \cap \mathcal{R} = \{0\}$;
        \item[(iv)]\label{KosFor} (Kostant's Formula) $\nabla_Xf_i = \mathrm{R}(X,R_i)$, where $\mathrm{R}$ is the curvature tensor.
    \end{itemize}
\end{lemma}
\begin{proof} The first three items are immediate consequences of Equation \ref{KosRev}. As for (iv), note that the tensor field $\mathcal{L}_{R_i}\nabla$ is zero since the flow of $R_i$ consists of isometries, which preserve the connection. Hence
\begin{align*}
    0 &= (\mathcal{L}_{R_i}\nabla)_XY \\
    &= \mathcal{L}_{R_i}(\nabla_XY) - \nabla_{\mathcal{L}_{R_i}X}Y - \nabla_X(\mathcal{L}_{R_i}Y) \\
    &= [R_i, \nabla_XY] - \nabla_{[R_i, X]}Y - \nabla_X[R_i, Y] \\
    &= \nabla_{R_i}\nabla_XY - \nabla_{\nabla_XY}R_i - \nabla_{[R_i, X]}Y - \nabla_X\nabla_{R_i}Y + \nabla_X\nabla_YR_i \\
    &= \nabla_{R_i}\nabla_XY - \nabla_{[R_i, X]}Y - \nabla_X\nabla_{R_i}Y + \nabla_Xf_i(Y) - f_i(\nabla_XY) \\
    &= \mathrm{R}(X,R_i)Y - (\nabla_Xf_i)Y.
\end{align*}
\end{proof}

We consider the Lie algebra $\mathfrak{e}$ of the enveloping group
\[
    E = \mathrm{E}^1_F = \overline{F(\mathbb{R}^q)} =  \overline{\mathrm{Span}\{\exp(tR_i); t \in \mathbb{R}, 1 \leq i \leq q\}}.
\]
Due to $C^1$-equicontinuity, it follows that $E$ is a torus and $\mathfrak{e}$ is a subalgebra of the family of Killing fields $\mathfrak{iso}(M)$ containing $\Gamma(\mathcal{R})$, the space of fields tangent to the Reeb distribution $\mathcal{R}$. 
\par For each point $p \in M$, the isotropy subalgebra $\mathfrak{e}_p$ is defined as the set of elements of $\mathfrak{e}$ whose value at $p$ is zero, i.e., the fields in $\mathfrak{e}$ whose flow fixes $p$.
This is exactly the Lie subalgebra associated with the isotropy subgroup $E_p \subset E$ of $p$.
Recall the type of the orbit $E(p) = \overline{\mathcal{F}(p)}$ is the conjugacy class of its isotropy group $E_p$.
Now, $E$ is a torus (hence an Abelian group) and therefore each such class has exactly one element. 
In other words, there is exactly one isotropy subgroup in $E$ for each orbit type.
On the other hand, since $M$ is compact, it follows from the Slice Theorem that there are only finitely many orbit types (see Theorem I.2.1 and Proposition I.2.4 in \cite{audin2012torus}, for instance).
Consequently, there are only finitely many distinct subalgebras $\mathfrak{e}_p$. 
\par Note that $\mathfrak{e}_p$ has dimension at most $\dim E - q$ since the Reeb fields never belong to it.
We want to choose a vector field $Z \in \mathfrak{e}$ avoiding subspaces of non-maximal dimension, which is a generic choice in the following sense. 
We write
\[
    \widetilde{\mathfrak{e}}_p = \mathfrak{e}_p \oplus \Gamma(\mathcal{R}) \subset \mathfrak{e},
\]
and let $\mathfrak{b}$ be the (\textit{finite}) union of all ``bad'' subspaces $\widetilde{\mathfrak{e}}_p \neq \mathfrak{e}$, that is, those whose dimension is non-maximal.
We choose a Killing field $Z$ to satisfy
\begin{equation}\label{eq:condZ}
        Z \in \mathfrak{e}\setminus\mathfrak{b}.
\end{equation}
This subset is nonempty because the orbits of minimal dimension are associated to the isotropy subgroups (and equivalently, subalgebras) of maximal dimension (cf. \cite[Section I.2]{audin2012torus} and \cite[Section I.4]{bredon1972introduction}).
The foliation $\mathcal{F}$ has at least two closed leaves \cite{finamore_contact_2022}, which are minimal orbits under the action of $E$.
For example, if all the leaves of $\mathcal{F}$ are closed, then $\dim E = q$ and therefore $\mathfrak{b}$ is empty.
In this situation, condition (\ref{eq:condZ}) above is trivially satisfied for any field.
\par Since $\mathcal{L}_Z\lambda_i = 0$ for any $i$, we have $d(\iota_Z\lambda_i) = -\iota_Zd\lambda_i$.
Thus the critical point set $C$ of the map
\begin{align*}
    S: M &\to \mathbb{R}\\
    p &\mapsto \iota_Z\lambda_i(p)
\end{align*}
 is exactly the set of points where $d\lambda_i(Z_p, \cdot) \equiv 0$, that is, the set
 \begin{equation}\label{eq:firstCChar}
    C = \{p \in M; Z_p \in \mathcal{R}_p\}.
 \end{equation}
Remark that $Z$ is an $\mathcal{F}$-foliated field, i.e., its flow preserves the foliation $\mathcal{F}$. 
In particular, if $Z$ is tangent to $\mathcal{R}$ at a point $p$, then it must be tangent to $\mathcal{R}$ at every point of $\mathcal{F}(p)$, and $\mathcal{F}{p}$ is invariant under the flow of $Z$ (see the proof \cite[Theorem 3.19]{finamore_contact_2022}, for instance).
This means $C$ is a $\mathcal{F}$-saturated subset of $M$.
\par Now, in light of our choice of $Z$, the dimension of the orbit $E(p)$ of a point $p \in C$ can not be greater than $q$.
Together with the equality on (\ref{eq:firstCChar}), this allows us to further characterise $C$ as
\begin{align*}
    C &= \{p \in M; \dim E(p) = q\} \\
      &= \{p \in M; \mathcal{F}(p) \text{ is closed}\}.
 \end{align*}
Thus, $C$ is precisely the union of the closed leaves of the contact foliation $\mathcal{F}$. 
Moreover, $C$ is the union of fixed points of the subtori $T \subset E$ whose dimension is $\dim E - q$, and therefore a submanifold of $M$.

\begin{remark}\label{rem:notation}
    As shown in \cite[Theorem 3.23]{finamore_contact_2022}, the set $C$ has a decomposition $C = \cup_iN_i$, where each $N_i$ is a totally geodesic submanifold of $(M,g)$ of even codimension. 
    Each such $N_i$ is saturated under $\mathcal{F}$, and can therefore been seen as a foliated manifold $(N_i, \mathcal{F}\rvert_{N_i})$, where $\mathcal{F}\rvert_{N_i}$ is simply the collection of closed leaves passing through $N_i$.
    Clearly, we have $\mathcal{R} \subset TN_i$ for each such submanifold.
    Our choice of $Z$ in a maximal isotropy subalgebra of $\mathfrak{e}$ ensures that, given a point $p\in N_i$, there is a neighbourhood of $p$ in $M$ where all the zeros of $Z$ belong to $N_i$. 
\end{remark}

\begin{lemma}
    Let $N$ be a connected component of $C$ and $p \in N$. Consider the Killing field
    \[
        K = Z - \sum_{i=1}^q\lambda_i(Z)R_i,
    \]
    and the tensor field
    \[
        \Phi = \sum_{j=1}^q(\iota_Z\lambda_j)f_j,
    \]
    where the $f_i$ are the associated tensor fields of Definition \ref{associatedTensorField}.
    Then for all $v, w \in T_pM$ perpendicular to $N$ we have:
    \begin{itemize}
        \item[(i)] $\nabla_vK = \nabla_vZ - \Phi(v)$ is a non-zero vector perpendicular to $N$. In particular, $\nabla_vK \in \xi_p$.
        \item[(ii)] $\mathrm{Hess}(S)\rvert_p(v,w) = 2g(\mathrm{R}(R_i,v)Z_p + f_i(\nabla_vZ), w).$
        \item[(iii)] The Hessian of $S$ along $N$ is non-degenerate in directions orthogonal to $N$.
    \end{itemize}
\end{lemma}

\begin{proof}
We remark that $\mathcal{R} \subset TN$, and since $\mathcal{R}$ and $\xi$ are orthogonal as per choice of $g$, it follows that $T^\perp N \subset \xi$.
To prove (i), first note that the Lie algebra of $E$ has a decomposition $\mathfrak{e} = \mathfrak{e}_p \oplus \Gamma(\mathcal{R})$, and $Z = K + \sum_i\lambda_i(Z)R_i$ is simply the corresponding decomposition of $Z$.
Evaluating the connection in the direction of $v$ yields
\begin{align}\label{EqDelVK1}
    \nabla_vZ &= \nabla_vK + \nabla_v\sum_{j=1}^q\lambda_j(Z)R_j \notag \\ 
    &= \nabla_vK + \sum_{j=1}^q(v\lambda_j(Z) R_j +  \lambda_j(Z)f_j(v)) \notag \\
    &= \nabla_vK + \Phi(v) + \sum_{j=1}^q(v\lambda_j(Z))R_j.
\end{align}
    Now notice that on $N$ the vector $Z$ is tangent to $\mathcal{R}$, and since $v \in \xi$, it follows that $\lambda_j([v, Z]) = 0$ for every $j$. 
    Hence
    \[
        v\lambda_j(Z) = d\lambda_j(v,Z) + Z\lambda_j(v) + \lambda_j([v, Z])= d\lambda_j(v,Z) = 0,
    \]
    and therefore Equation \ref{EqDelVK1} becomes $\nabla_vK = \nabla_vZ - \Phi(v)$, as wanted.

    We claim $\nabla_vK \neq 0$. To see this, suppose by contradiction that $\nabla_vK = 0$ and consider a geodesic $\gamma$ starting at $p$ with velocity vector $v$. Since $K$ is Killing, it is a Jacobi field for $\gamma$, and since $K_p = 0$, it follows that $K\rvert_\gamma = 0$. On the other hand, our choice of $Z$ implies that, in a neighbourhood of $N$, the only zeros of $Z$ are those in $N$. Thus $\gamma \subset N$, which contradicts the orthogonality of $v$ to $N$. Therefore $\nabla_vK \neq 0$.

    Now, to see that $\nabla_vK$ is orthogonal to $N$, we note that $N$, being the zero set of a Killing field, is a totally geodesic submanifold. Together with the facts that $K$ is Killing and $K\rvert_N$ is tangent to $N$, this implies $\nabla_XK$ tangent to $N$ for any $X$ tangent to $N$. Thus
    \[
        g(\nabla_vK, X) = - g(\nabla_XK,v) = 0,
    \]
    and since $X$ was arbitrarily chosen, it follows that $\nabla_vK \in T^\perp N \subset \xi.$

    As for (ii), we consider fields $V$ and $W$ extending $v$ and $w$. Suppose these fields are obtained employing parallel transporting $v$ and $w$ along the geodesics starting at $p$. Observe that $R_i$ and $Z$ are commuting Killing fields, hence 
    \[
        \nabla_{R_i}Z = f_i(Z) + [R_i, Z] = f_i(Z).
    \]
    Using the relation above, the properties of our choice of $Z$, and Lemma \ref{KosFor}(iv), we obtain
    \begin{align*}
    \mathrm{Hess}(S)\rvert_p(v,w) &= V(WS(p)) \\
                                  &= V(W\lambda_i\rvert_p(Z)) \\
                                  &= V(Wg_p(R_i, Z) \\
                                  &= V(g_p(f_i(W), Z) + g_p(R_i, \nabla_WZ)) \\
                                  &= V(-g_p(f_i(Z), W) - g_p(f_i(Z), W) \\
                                  &= -2Vg_p(f_i(Z), W) \\
                                  &= -2(g_p(\nabla_Vf_i(Z), W) + g_p(f_i(Z), \nabla_VW))\\
                                  &= -2( g_p((\nabla_Vf_i)(Z) + f_i(\nabla_VZ), W) + \frac{1}{2}d\lambda_i(Z, \nabla_VW)\\
                                  &= -2g(\mathrm{R}(v, R_i)Z + f_i(\nabla_vZ), w).
    \end{align*}
    Finally, to see that $\mathrm{Hess}(S)\rvert_p$ is non-degenerate in the normal bundle of $N$, we use Kostant's Formula and observe that
    \[
        \mathrm{R}(X, R_i)R_j = (\nabla_Xf_i)R_j = \nabla_X(f_i(R_j)) - f_i(\nabla_XR_j) = -f_if_j(X).
    \]
    Now we use item (ii) with $w = f_i(\nabla_vK)$, yielding:
    \begin{align*}    
        \mathrm{Hess}(S)\rvert_p(v,&f_i(\nabla_vK)) = -2g(\mathrm{R}(v, R_i)Z + f_i(\nabla_vZ), f_i(\nabla_vK)) \\
                                                   &= -2g(R(v, R_i)Z, f_i(\nabla_vK)) + g(+ f_i(\nabla_vZ), f_i(\nabla_vK)) \\
                                                   &= -2g\left(\sum_{j=1}^q\iota_Z\lambda_j\mathrm{R}(v, R_i)R_j + f_i(\nabla_vZ), f_i(\nabla_vK)\right) \\
                                                   &= -2g\left(-\sum_{j=1}^q\iota_Z\lambda_jf_if_j(v) + f_i(\Phi(v) + \nabla_vK), f_i(\nabla_vK)\right) \\
                                                   &= -2g\left(-f_i(\Phi(v)) + f_i(\Phi(v)) + f_i(\nabla_vK), f_i(\nabla_vK)\right) \\
                                                   &= -2\|(f_i(\nabla_vK), f_i(\nabla_vK)\|
    \end{align*}
    Now note that $f_i(\nabla_vK) \neq 0$. In fact, suppose by contradiction that $f_i(\nabla_vK) = 0$. Then, for every $X \in \xi$:
    \[
        0 = g(X, f_i(\nabla_vK)) = - g(f_i(X), \nabla_vK) = -\frac{1}{2}d\lambda_i(X, \nabla_vK),
    \]
    which can not happen since $d\lambda_i$ is non-degenerate on $\xi.$
    This means for each $v$ normal to $N$ there is another vector $f_i(\nabla_vK)$ such that $\mathrm{Hess}(S)\rvert_p(v,f_i(\nabla_vK)) \neq 0$, that is, $\mathrm{Hess}(S)\rvert_p$ is non-degenerate in normal directions.
\end{proof}

Recall that the Betti number $b_i(M;\mathbb{R})$ is the dimension of the $i$-th cohomology group $H^i_{dR}(M)$ over $\mathbb{R}$. Similarly, we define the $\mathcal{F}$-basic Betti number $b_i(\mathcal{F})$ to be the real dimension of the basic cohomology group $H_B^i(\mathcal{F})$. The $\mathcal{F}$-basic Poincaré polynomial of the foliated manifold $(M, \mathcal{F})$ is 
\[
    P_{\mathcal{F}}(t) = \sum_{i = 0}^{\mathrm{codim}(\mathcal{F})}t^ib_i(\mathcal{F}).
\]
\par the fact that $S: M \to \mathbb{R}$ is a $\mathcal{F}$-basic Morse-Bott function allow us to relate the $\mathcal{F}$-basic Poincaré polynomials of the foliated manifolds $(M, \mathcal{F})$ and $(N, \mathcal{F}\rvert_N)$ (see Remark \ref{rem:notation}).
In order to do that, we apply the results from \cite{goertsches_equivariant_2018} to each connected component $N$ of the critical set $C$, getting:
\begin{theorem} The $\mathcal{F}$-basic Poincaré polynomials for $M$ and $N$ satisfy
    \[
        P_{\mathcal{F}}(t) = \sum_Nt^{i_N}P_{\mathcal{F}\rvert_N}(t),
    \]
    where $N$ runs over all the connected components of $C$, and $i_N$ is the index of $N$, i.e., the rank of the negative normal bundle of $N$ with respect to $S$.
\end{theorem}
 In particular, by evaluating both sides at $t = 1$, we obtain
 \begin{theorem}\label{Morse-DeRham_isomorphism} Let $(M, \vec{\lambda}, \mathcal{R} \oplus \xi)$ be an isometric contact foliation, and, for $Z$ a vector field chosen as above, consider the critical point set $C$ of the Morse-Bott function $S: p \mapsto \lambda_i(Z_p)$. Then
     \[
        \dim_\mathbb{R}H^\ast_B(\mathcal{F}) = \dim_\mathbb{R}H^\ast_B( \mathcal{F}\rvert_C).
     \]
     In particular, if $\mathcal{F}$ has only finitely many closed orbits, then the dimension of the basic cohomology ring $H^\ast_B(\mathcal{F})$ is exactly the number of closed orbits of $\mathcal{F}$.
 \end{theorem}
Theorem \ref{Morse-DeRham_isomorphism} allows us to estimate the number of closed orbits by studying the basic cohomology of the contact foliation $\mathcal{F}$. The first thing to note is that the even-dimensional basic cohomology groups are never zero because the exterior derivatives of the adapted coframe $\vec{\lambda}$ define non-zero basic forms. It is known that an invariant transversal volume form $\mu$ for a harmonic foliation $(M, \mathcal{F})$ represents a non-zero class on the top-dimensional basic cohomology space $H^{\mathrm{codim}\mathcal{F}}_B(\mathcal{F})$. 
This implies, in particular, that if $\mathcal{F}$ is a foliation of even codimension $2n$, and $\omega$ is an invariant transversal symplectic form, then $[\omega]^i \neq 0 \in H^{2i}_B(\mathcal{F})$ for $i = 1, \cdots, n$ (cf. \cite[Theorems 4.32 and 4.33]{TondeurFoliationsRiemannianManifolds2012}). 
In the case of $q$-contact structures, one notes that the operator 
\[
    \omega \mapsto \int_M\lambda\wedge\omega,
\]
where $\lambda := \lambda_1 \wedge \cdots \wedge \lambda_q$ is the characteristic form, descends to an operator $H^\ast_B(\mathcal{F}) \to \mathbb{R}$, because the $q$-form $\lambda$ is $\mathcal{F}$-closed (cf. \cite{finamore_contact_2022}). Such operator maps $[d\lambda_i]^n$ to a non-zero number, since $\lambda\wedge(d\lambda_i)^n$ is a volume form on $M$, hence $[d\lambda_i]^n \neq 0$, and consequently $[d\lambda_i]^j \neq 0$ for every $j = 1, \cdots, n$.

To better use this fact, we associate with each adapted coframe the following quantity.
\begin{definition}
Let $(M, \vec{\lambda}, \mathcal{R} \oplus \xi, g)$ be an isometric contact foliation on the $(2n+q)$-dimensional manifold $M$. Define $\delta_0(\vec{\lambda}) := 1$, and for $i = 1, \cdots, n$, let $\delta_i(\vec{\lambda}))$ be the dimension of the linear subspace of $H^{2i}_B(\mathcal{F})$ spanned by $\{[d\lambda_1]^i, \cdots, [d\lambda_q]^i\}$. 
In other words
\[
\delta_i(\vec{\lambda}) := \max\{\# L; L \subset \{[d\lambda_1]^i, \cdots, [d\lambda_q]^i\} \text{ is linearly independent}\}.
\]
The natural number 
\[
    \delta(\vec{\lambda)} := \sum_{i = 0}^n \delta_i(\vec{\lambda}) = 1 + \sum_{i = 1}^n \delta_i(\vec{\lambda})
\] 
is the \textbf{basic dimension} of the adapted coframe $\vec{\lambda}$.
\end{definition}

Note that $\delta_i(\vec{\lambda})$ is bounded below by $1$ and above by either $q$ or the dimension of $H^{2i}_B(\mathcal{F})$. Hence the basic dimension satisfies the inequalities
\begin{equation}\label{basis_dimension}
    n + 1 \leq \delta(\vec{\lambda}) \leq \min\{qn + 1, \dim H^{2\ast}_B(\mathcal{F})\}. 
\end{equation}
We remark that for uniform contact foliations, the basic dimension is always minimal, i.e., equal to exactly $n+1$. Using Theorem \ref{Morse-DeRham_isomorphism} we obtain
 \begin{theoremB}\label{ClosedOrbitsLowerBound}
     Let $\mathcal{F}$ be an $C^1$-equicontinuous contact foliation on a closed manifold $M$. Then the number of closed orbits of $\mathcal{F}$ is at least $\delta(\vec{\lambda}) + b_1(M ; \mathbb{R})$. In particular, an isometric contact foliation of codimension $2n$ has no less than $n+1$ closed orbits.
 \end{theoremB}
 \begin{proof}
     By construction, we have $\delta_i(\vec{\lambda}) \leq b_{2i}(\mathcal{F})$. 
     Thus it is immediate that the basic dimension of the adapted coframe is a lower bound for $\dim_\mathbb{R}H^{2\ast}_B(\mathcal{F})$. Note that this bound does not take into account the dimensions of any of cohomology groups $H^i_B(\mathcal{F})$ for odd $i$. On the other hand, it was shown in \cite[Theorem 3.10]{finamore_contact_2022} that every harmonic $1$-form on an isometric $q$-contact manifold is also $\mathcal{F}$-basic, implying an isomorphism $H^1_{dR}(M) \approx H^1_B(\mathcal{F})$ (cf. \cite[Proposition 3.11]{finamore_contact_2022}). In particular, $b_1(M;\mathbb{R}) = b_1(\mathcal{F})$, and therefore $\delta(\vec{\lambda}) + b_1(M;\mathbb{R})$ is a lower bound for $\dim_\mathbb{R}H^\ast_B(\mathcal{F})$.
 \end{proof}

     The following is a direct application of Theorem B.
     
     \begin{theorem}
         On a closed manifold, every quasiconformal Reeb field satisfies the Weinstein conjecture. Moreover, if the dimension of the ambient manifold is $2n+1$, then the Reeb field has at least $n+1$ closed orbits.
     \end{theorem}
     As a corollary, we present a new proof of a result of Banyaga's.

\begin{theorem}[Theorem 1 in \cite{banyaga_note_1990}] 
    Let $M$ be a closed manifold and $\lambda$ a regular contact form on $M$. If $\lambda'$ is $C^2$-close to $\lambda$, then $(M,\lambda')$ satisfies the Weinstein conjecture, and its Reeb field has at least $2$ closed orbits.    
\end{theorem}

\begin{proof}
    The contact condition for $1$-forms is simply that $\lambda\wedge d\lambda > 0$, which is clearly open in the $C^1$-topology. On the other hand, quasiconformality is also an open property. As seen in Example \ref{regularIsQsc}, the Reeb field associated to $\lambda$ is quasiconformal. If $\lambda$ and $\lambda'$ are sufficiently $C^2$-close, then their contact actions $F_\lambda$ and $F_{\lambda'}$ are $C^1$-close enough to guarantee that $F_{\lambda'}$ is also quasiconformal. Therefore, $(M, \lambda')$ satisfies the Weinstein conjecture, due to Theorem \ref{ClosedOrbitsLowerBound}. In particular, the Reeb field must have at least $2$ distinct closed orbits.
\end{proof}

  As remarked before, the basic dimension does not take into account any of the odd-dimensional basic cohomology groups. Hence, one should not expect it to be equal to the number of closed orbits. 

 \begin{example}\label{StiefelProduct}
     We consider the $7$-dimensional Stiefel manifold 
     \[
            V_{2,5} \approx ~^{\textstyle \mathrm{SO}(5)}\!\big/_{\textstyle \mathrm{SO}(3)}.
     \]
     As shown in \cite[Section 8]{goertsches_equivariant_2012}, $V_{2,5}$ supports a K-contact structure $(\alpha, g)$ whose Reeb field $S$ has exactly $4 = \delta(\alpha)$ orbits. We consider the $14$-dimensional manifold $M = V_{2,5} \times V_{2,5}$, and denote by $\pi_i: M \to V_{2,5}$ the projection on the $i$-th coordinate. For $i = 1,2$, we write
     \begin{align*}
         g_i &:= \pi_i^\ast g, \\
         \sigma_i &:= \pi_i^\ast\alpha, \\
         X_i &:= \pi_i^\ast S.         
     \end{align*}
     Then the $1$-forms
     \begin{align*}
        \lambda_+ &= \sigma_1 + \sigma_2 \\
        \lambda_- &= \sigma_1 - \sigma_2
     \end{align*}
     define an non-uniform $2$-contact structure on $M$ whose Reeb fields are $R_+ = 2^{-1}(X_1 + X_2)$ and $R_- = 2^{-1}(X_1 - X_2)$, respectively, as in Example \ref{product}. The metric $g = g_1 + g_2$ on $M$ is such that each Reeb field $R_i$ is Killing, so that $(M, \{\lambda_+, \lambda_-\}, g)$ defines an isometric $2$-contact structure.  The contact foliation $\mathcal{F}$ is the product of the contact flows in $(V_{2,5}, \alpha)$. In particular, the closed leaves are exactly the products of closed flow-lines. Hence $(M, \mathcal{F})$ has exactly $16$ closed leaves, and $\dim_\mathbb{R}H^\ast_B(\mathcal{F})$ is $16$, according to Theorem \ref{Morse-DeRham_isomorphism}. On the other hand, the Stiefel manifold $V_{2,5}$ is a real cohomology sphere of dimension $7$. Hence the Künneth formula implies that the first cohomology group of $M$ is $0$. Moreover, $\mathcal{F}$ has codimension $12 = 2\cdot 6$, hence the minimal number of closed leaves of $\mathcal{F}$ as given by Theorem \ref{ClosedOrbitsLowerBound} is 7. Using Equation \ref{basis_dimension}, we conclude that the basic dimension of the coframe $\{\lambda_+, \lambda_-\}$ is bounded above by $2\cdot 6 + 1 = 13$, so that is the maximum number of closed orbits one could assume $\mathcal{F}$ has by using the estimates of Theorem \ref{ClosedOrbitsLowerBound} alone.
 \end{example}
 Observe that the basic dimension $\delta(\{\lambda_+, \lambda_-\})$ is not minimal. In general, for an adapted coframe $\{\lambda_1, \cdots, \lambda_q\}$, two basic classes $[d\lambda_i]$ and $[d\lambda_j]$ satisfy an equality 
 \[
    [d\lambda_i]^l = a[d\lambda_j]^l
 \]
 for a non-zero real number $a$ if and only if there is a basic $(2l-1)$-form $\eta$ such that 
 \[
    (d\lambda_i)^l-a(d\lambda_j)^l = d\eta,
 \]
 and therefore
 \begin{equation}\label{BasicClassesLinearlyDependent}
    \lambda_i\wedge(d\lambda_i)^{l-1} - a\lambda_j\wedge(d\lambda_j)^{l-1} = \eta + \theta,
 \end{equation}
 where $\theta$ is a closed $(2l-1)$-form. In particular, in the case of the manifold $M$ of Example \ref{product}, if we assume $\delta_1(\{\lambda_+, \lambda_-\}) = 1$ and apply Equation \ref{BasicClassesLinearlyDependent} we get
 \begin{equation}\label{ProductIsNonUniform}
    \lambda_+ - a\lambda_- = df + \theta,
 \end{equation}
 where $f$ is a basic function and $d\theta = 0$, because $H_B^1(\mathcal{F}) \approx H^1_{dR}(M) = 0$ (cf. \cite[Proposition 3.11]{finamore_contact_2022}). Taking exterior derivatives in Equation \ref{ProductIsNonUniform} we obtain
 \[
    (a-1)d\sigma_1 = -(a+1)d\sigma_2,
 \]
 which can not happen for any real $a$. Thus $\delta_1(\{\lambda_+, \lambda_-\}) = 2$ and consequently $\delta(\{\lambda_+, \lambda_-\}) = 13$ is strictly bigger than $7 = 2^{-1}\mathrm{codim}(\mathcal{F})+1$.

\section{Cohomology of isometric contact foliations}\label{cohom}

As shown in Example \ref{StiefelProduct}, the lower bound of $2^{-1}\mathrm{codim}(\mathcal{F}) - 1$ need not be the exact number of closed orbits, even when there are only finitely many of them. However, in the case when the number of closed orbits is finite and \emph{minimal}, i.e., equal to exactly $2^{-1}\mathrm{codim}(\mathcal{F}) - 1$, then substantial topological restrictions are imposed on the ambient manifold. This is true for metric $f$-K-contact structures, and since we showed in Section \ref{morseTheory} that the function $S: M \to \mathbb{R}$ is Morse-Bott even without the presence of the tensor $f$, those topological restrictions carry over to the more general class of isometric uniform $q$-contact manifolds. All the proofs in \cite{goertsches_topology_2020} apply, \emph{mutatis mutandis}, to the case of $q$-contact manifolds, giving rise to the results in this section.

\begin{lemma}\cite[Proposition 4.4]{goertsches_topology_2020}\label{exactSequencesBasicCohomologies}
    Let $\mathcal{F}$ be a $q$-dimensional contact foliation, and $\mathcal{F}_s$ be the integral foliation of the bundle spanned by the Reeb fields $R_1, \cdots, R_i$. If $\mathcal{F}$ is $C^1$-equicontinuous, then for $s = 0, \cdots, q-2$ there are short exact sequences
    \[
        0 \longrightarrow H^\ast_B(\mathcal{F}_{s+1}) \longrightarrow H^\ast_B(\mathcal{F}_s) \longrightarrow H^{\ast-1}_B(\mathcal{F}_{s+1}) \longrightarrow 0,
    \]
    as well as a long exact sequence    
     \[
     \scalebox{.86}{$
        \cdots \xrightarrow{} H^\ast_B(\mathcal{F}) \xrightarrow{} H^\ast_B(\mathcal{F}_{q-1}) \xrightarrow{} H^{\ast-1}_B(\mathcal{F}) \xrightarrow{\delta} H^{\ast+1}_B(\mathcal{F})\xrightarrow{} \cdots
    $}
    \]    
    where $\delta([\sigma]) = [d\lambda_{q-1}\wedge\sigma]$.
\end{lemma}

Now, if the isometric $q$-contact structure is also uniform, that is, if $d\lambda_i = \omega$ for every $i$, then the first de Rham cohomology group of $M$ has dimension at least $q-1$ since each form $\lambda_i-\lambda_q$ is closed and non-exact \cite[Proposition 2.36]{finamore_contact_2022}. 
In particular, we can obtain a homomorphism $\wedge(\mathbb{R}^{q-1}) \to H^1_{dR}(M)$, where $\wedge(\mathbb{R}^{q-1})$ is the exterior algebra on $q-i$ generators, by sending the $i$-th generator to the class $[\lambda_i - \lambda_q]$. This gives $H^\ast_{dR}(M)$ the structure of an $\wedge(\mathbb{R}^{q-1})$-algebra, and then the exact sequences of Lemma \ref{exactSequencesBasicCohomologies} allow us to conclude the existence of the following isomorphism (cf. \cite[Theorem 1.1]{goertsches_topology_2020})
\begin{equation}
    H^\ast_{dR}(M) \approx \wedge(\mathbb{R}^{q-1}) \otimes H^\ast_B(\mathcal{F})
\end{equation}
Using the isomorphism above, we obtain
\begin{theorem}\cite[Theorem 6.4]{goertsches_topology_2020}
    For a uniform $C^1$-equicontinuous $q$-contact foliation $\mathcal{F}$ on a $(2n+q)$-dimensional closed manifold $M$, the following are equivalent.
    \begin{itemize}
        \item[(i)] $\mathcal{F}$ has exactly $n+1$ closed orbits;
        \item[(ii)] The basic cohomology of $\mathcal{F}$ is the same as the de Rham cohomology of $\mathbb{C}P^n$;
        \item[(iii)] The basic cohomology of $\mathcal{F}_{q-1}$ is the same as the de Rham cohomology of the sphere $S^{2n+1}$;
        \item[(iv)] The de Rham cohomology of $M$ is the same as that of $S^{2n+1} \times \mathbb{T}^{q-1}$.
    \end{itemize}
\end{theorem}

Concerning item (iv), we remark that a general uniform $q$-contact foliation is known to be a fibration over the torus $\mathbb{T}^{q-1}$ \cite[Theorem 2.38]{finamore_contact_2022}, while a similar result of Goertsches and Loiudice \cite{goertsches_how_2020} states that every metric $f$-K-contact manifold can be constructed from a K-contact manifold by taking mapping tori and applying certain deformations.

\section{Conclusions and further research.}

In conclusion, we have established the equivalence between different notions of regularity and the existence of invariant transverse geometric structures for contact foliations.
Specifically, we now know that $C^1$-equicontinuity, the existence of an invariant metric, the existence of a bundle-like metric, quasiconformality, and the existence of an invariant conformal structure are all equivalent properties of a contact foliation on a compact manifold.
\par Furthermore, we have extended existing arguments and constructions to a broader class of contact-like structures, which allowed us to significantly improve the lower bound on the number of closed orbits for isometric contact foliations. Specifically, we have increased this lower bound from $2$ to $n+1$ in ambient codimension $2n$.
\par In future studies, we would like to further explore the relationship between quasiconformality and tightness in contact structures.
Based on our results exposed in Theorem \ref{thm: otd&qsc}, we conjecture that every quasiconformal Reeb field must be necessarily associated to a tight contact structure. Additionally, we plan to examine the applicability of our findings to more general contact-like structures, such as the $k$-contact structures defined in \cite{gaset_contact_2020}.
\par Lastly, we are currently focusing on applications of the theory of $q$-contact structures. Specifically, we believe that quasiconformal contact actions can be utilised to classify for quasiconformal Anosov actions, extending classification theorems for Anosov flows found in Fang's seminal work \cite{fang_smooth_2004}.


\providecommand{\bysame}{\leavevmode\hbox to3em{\hrulefill}\thinspace}
\providecommand{\MR}{\relax\ifhmode\unskip\space\fi MR }
\providecommand{\MRhref}[2]{%
  \href{http://www.ams.org/mathscinet-getitem?mr=#1}{#2}
}
\providecommand{\href}[2]{#2}

\end{document}